\documentclass[11pt]{article}
\usepackage{amssymb, amsmath, amsthm, amsfonts,enumerate,appendix}
\usepackage{amsmath,setspace,scalefnt}
\usepackage[dvipsnames,svgnames,table]{xcolor}
\usepackage{graphicx,tikz,caption,subcaption}
\usepackage{fullpage}
\usepackage{hyperref}
\usepackage{relsize}
\usepackage{enumitem,verbatim}
\usetikzlibrary{decorations.pathmorphing}
\usepackage[left=2.2cm, right=2.2cm, top=2.1cm, bottom=2.3cm]{geometry}

\makeatletter
\newcommand{\labelinthm}[1]{%
   \label{temp#1}
   \protected@write \@auxout {}{\string \newlabel{#1}{{\emph{\ref{temp#1}}}{\thepage}{\emph{\ref{temp#1}}}{temp#1}{}} }%
}
\makeatother

\newcommand{\claimproofstart}[1][Proof]{\begin{proof}[#1]
\renewcommand{\qedsymbol}{$\boxdot$}}
\newcommand\claimproofend{
\end{proof}
\renewcommand{\qedsymbol}{$\square$}}

\tikzset{snake it/.style={decorate, decoration=snake}}
\usetikzlibrary{calc}

\newtheorem{theorem}{Theorem}[section]
\newtheorem{lemma}[theorem]{Lemma}

\newtheorem{conjecture}[theorem]{Conjecture}

\theoremstyle{definition}

\newtheorem{claim}{Claim}

\newcommand{\eps}{\varepsilon}

\renewcommand{\P}{\mathbb{P}}

\date{}

\title{Nearly-uniform degree distributions in spanning subgraphs}
\author{Richard Montgomery\thanks{Mathematics Institute, University of Warwick, Coventry, UK. Research supported by the European Research Council (ERC) under the European Union Horizon 2020 research and innovation programme (grant agreement No.\ 947978). Email: {\tt richard.montgomery@warwick.ac.uk}.}
 \and Alexey Pokrovskiy\thanks{Department of Mathematics, University College London, London, UK. {Email}: {\tt dralexeypokrovskiy@gmail.com.}} \and Benny Sudakov\thanks{
Department of Mathematics, ETH, Z\"urich, Switzerland. Research supported in part by SNSF grant 200021-228014. {Email}: {\tt benjamin.sudakov@math.ethz.ch}.
}}

\begin{document}
\maketitle

\begin{abstract}
We show that, when $d=o(n)$, every $d$-regular $n$-vertex graph
contains a spanning subgraph whose degree distribution is nearly
uniform, i.e., for each $0\leq i\leq d$, there are
$(1+o(1))n/(d+1)$ vertices with degree $i$. This proves a conjecture of Alon and Wei on irregular subgraphs and strengthens a previous result of Fox,
Luo and Pham.
\end{abstract}

\section{Introduction}\label{sec:intro}

Many problems in graph theory fit into the following general meta-problem.
Any weight distribution on the edges of a graph $G$ produces a natural weighted degree distribution on its vertices
via the total weight on the neighbouring edges of each vertex. Can, then, edge weights be chosen for $G$ (often from a small set of integers) so that the corresponding weighted degrees satisfy some prescribed property? Desirable properties include that adjacent vertices receive different weighted degrees, that all the weighted degrees are distinct, that the
weighted degrees belong to prescribed sets, or that the distribution of
weighted degrees has a particular form.

One prominent example is the 1--2--3-Conjecture of Karo\'nski,
{\L}uczak and Thomason~\cite{karonski2004edge} from 2004. That is, can 
the edges of a graph be assigned weights from $\{1,2,3\}$ so that any
adjacent vertices receive different weighted degrees?  This problem attracted a great deal of attention (see, e.g.,~\cite{addario2007vertex,kalkowski2010vertex}), until it was recently
solved when Keusch~\cite{keusch2024solution} confirmed that such an
assignment exists for every graph without isolated edges.

The topic of the 1--2--3-Conjecture can be interpreted as a local version of the more classical notion of \emph{irregularity strength}, itself a more global example of edge-to-vertex weighting problems. 
Here one asks for an
edge weighting which distinguishes all of the vertices via their weighted degrees simultaneously. As introduced by Faudree, Schelp, Jacobson and Lehel~\cite{faudree1989irregular} in 1987, the
irregularity strength $s(G)$ of a graph $G$ is the least integer $k$
such that the edges of $G$ can be assigned weights from $\{1,\dots,k\}$
so that all of the resulting weighted degrees are distinct (taking $s(G)=\infty$ if no such $k$ exists).  
For a $d$-regular graph on $n$ vertices, a simple counting argument gives
$s(G)\geq n/d$, and Faudree and Lehel~\cite{faudree1987bound}
conjectured in 1987 that this lower bound is sharp up to an additive constant when $d\geq 2$: i.e., 
there is an absolute constant $c$ such that $s(G)\leq n/d+c$ for every
$d$-regular graph $G$.  This conjecture has attracted considerable
attention, with recent progress including an asymptotic solution by Przyby{\l}o and Wei~\cite{przybylo2023general,przybylo2023short}.

The most classical setting for edge-to-vertex weighting questions is when the allowed edge
weights are $0$ and $1$.  In this case, the edge weightings are equivalent
to the spanning subgraphs of the graph, a correspondence given when the edges are weighted $1$ whenever they are included in the subgraph, and $0$ otherwise. The weighted degrees are
then precisely the ordinary degrees in the subgraph.  Thus the problem
becomes that of finding a spanning subgraph whose degrees obey some
specified restrictions. This setting, then, includes the theory of factors, going back
to the celebrated factor theorems of Tutte~\cite{Tutte1952Factors} and Lov\'asz~\cite{Lovasz1970Subgraphs}, as well as more
general degree-constrained subgraph problems.  One interesting example
is a result of Shirazi and Verstra\"ete~\cite{shirazi2008note}.  They
proved, using algebraic methods, that if each vertex $v$ in a graph $G$ is given a set
$F(v)\subset \{0,1,\dots,d_G(v)\}$ of size greater than $d_G(v)/2$,
then there is a spanning subgraph $H$ with $d_H(v)\in F(v)$ for every
$v\in V(G)$.

In this paper, we consider another very natural problem in the $0/1$
setting, as posed by Alon and Wei~\cite{alon2023irregular}.  Instead of
prescribing each individual vertex degree, one asks
for a spanning subgraph whose degree distribution is as uniform as
possible. Let $G$ be a $d$-regular graph on $n$ vertices, and let $H$
be a spanning subgraph of $G$.  Since every degree in $H$ lies in
$\{0,1,\dots,d\}$, the best possible outcome would be that these $d+1$ degree
values occur almost equally often.  Equivalently, how close can the subgraph $H$ come to being maximally irregular, i.e., to having its vertex degrees spread evenly over all available degrees?  Alon and Wei
made the following conjecture.

\begin{conjecture}\label{conj:AWexact}
	Every $d$-regular graph on $n$ vertices contains a spanning subgraph
	$H$ such that
	\[
	\left|\#\{v\in V(H):d_H(v)=k\}-\frac{n}{d+1}\right|\leq 2
	\qquad\text{for all }0\leq k\leq d.
	\]
\end{conjecture}

\noindent As pointed out by P\'{a}lv\"{o}lgyi (see~\cite{alon2023irregular}),
this would be tight.  Alon and Wei~\cite{alon2023irregular} also
proposed the following asymptotic version of Conjecture~\ref{conj:AWexact}.

\begin{conjecture}\label{conj:AWasympt}
	Let $d=o(n)$.  Every $d$-regular graph on $n$ vertices contains a
	spanning subgraph $H$ such that
	\[
	\#\{v\in V(H):d_H(v)=k\}=(1+o(1))\frac{n}{d+1}
	\qquad\text{for all }0\leq k\leq d.
	\]
\end{conjecture}

\noindent
The assumption $d=o(n)$ is necessary here: allowing
$d=\Omega(n)$ in Conjecture~\ref{conj:AWasympt} corresponds in this case to replacing the constant 2 in Conjecture~\ref{conj:AWexact} by $o(1)$, an impossible strengthening.

A natural approach to Conjecture~\ref{conj:AWasympt} is to use the
\emph{irregular random subgraph} introduced by Frieze, Gould,
Karo\'nski and Pfender~\cite{frieze2002graph} in the context of
irregularity strength.  To define the irregular random subgraph of a graph $G$, assign to each vertex $v\in V(G)$
an independent uniformly random variable $x_v\in[0,1]$, and include an
edge $uv\in E(G)$ precisely when
\[
x_u+x_v\geq 1.
\]
When $G$ is $d$-regular with $n$ vertices, this random subgraph has the useful property
that, for every $0\leq k\leq d$, the expected number of vertices of
degree $k$ is exactly $n/(d+1)$.  The results of Frieze, Gould,
Karo\'nski and Pfender~\cite{frieze2002graph} already implied
that Conjecture~\ref{conj:AWasympt} holds in the range
$d=o((n/\log n)^{1/4})$. Analysing the same random construction more carefully has proven fruitful. Indeed, using this construction, Alon and Wei~\cite{alon2023irregular} extended the range in which Conjecture~\ref{conj:AWasympt}
 was known to $d=o((n/\log n)^{1/3})$ and then Fox, Luo and
Pham~\cite{fox2024random} further extended it to
\[
d=o\left(\frac{n}{(\log n)^{12}}\right).
\]
In a different direction, Ma and Xie~\cite{ma2025finding} recently used
a deterministic local-adjustment method to make progress on the exact
version, Conjecture~\ref{conj:AWexact}, proving it with the constant $2$
replaced by $2d^2$. This gives another proof of Conjecture~\ref{conj:AWasympt}
in the range $d=o(n^{1/3})$.

Our main result proves the asymptotic conjecture of Alon and Wei in the
full range $d=o(n)$.

\begin{theorem}\label{thm:AWasympt}
	For each $\eps>0$, there is some $\lambda>0$ such that, for every
	$d\leq \lambda n$, every $d$-regular graph on $n$ vertices contains a
	spanning subgraph $H$ such that
	\begin{equation}\label{eq:balanceddegrees}
		\#\{v\in V(H):d_H(v)=k\}=(1\pm \eps)\frac{n}{d+1}
		\qquad\text{for all }0\leq k\leq d.
	\end{equation}
\end{theorem}

\noindent
Since $\eps>0$ is arbitrary, Theorem~\ref{thm:AWasympt} immediately
implies Conjecture~\ref{conj:AWasympt}. It would be interesting to
understand how small the error parameter $\eps$ can be made using our
methods, or refinements of them. 
On the other hand, Conjecture~\ref{conj:AWexact} remains open. If it is true, then obtaining such a flat distribution on the degrees in a subgraph seems likely to require significant additional structural ideas beyond our current techniques. 
This is because our methods are in part probabilistic, and here this introduces an error that appears very hard to correct locally.

The irregular random
subgraph is used in our proof, but, unlike the previous papers using it, we then modify it to get a graph satisfying \eqref{eq:balanceddegrees}.  Indeed, as
noted by Fox, Luo and Pham~\cite{fox2024random}, the irregular random
subgraph itself is not expected to satisfy \eqref{eq:balanceddegrees} in
the whole range $d=o(n)$.  Therefore, after taking a typical irregular
random subgraph $H_0\subset G$, we find a well-behaved correction function
$\psi:V(G)\to\mathbb Z$ such that the quantities
$d_{H_0}(v)-\psi(v)$ already have the desired nearly-uniform
distribution.  The function $\psi$ is chosen so that no vertex requires a
large correction, and so that a vertex whose degree should be decreased
has many available incident edges in $H_0$, while a vertex whose degree
should be increased has many available incident edges in $G\setminus
H_0$.

It remains to realise (something close to) these desired degree corrections by actually
removing and adding edges.  For this, we randomly orient the edges of $G$ and then use a random algorithm which (with high probability) changes the out-degree of each
vertex by the amount prescribed by $\psi$. The unavoidable side effects are
confined to in-degree changes at a carefully chosen small set of
vertices, and these changes are further restricted to a sparse set of possible
values.  In combination, this keeps the final degree distribution nearly uniform.  A
more detailed proof sketch is given in Section~\ref{sec:sketch}, and
the proof of Theorem~\ref{thm:AWasympt} is completed in
Section~\ref{sec:mainproof}.

\section{Preliminaries}

In this section, we will sketch our proof of Theorem~\ref{thm:AWasympt}, before giving an overview of our notation and then recalling the concentration bounds we will use. 

\subsection{Proof sketch for Theorem~\ref{thm:AWasympt}}\label{sec:sketch}
Let $G$ be a $d$-regular $n$-vertex graph. We wish to find a spanning subgraph $H\subset G$ with $(1+o(1))n/(d+1)$ vertices of degree $k$, for each $k\in [0,d]$. Our starting point is the irregular random subgraph $H_0\subset G$, where we include any edge $uv\in E(G)$ if $x_u+x_v\geq 1$, using the independent uniform random variables $x_v\in [0,1]$, $v\in V(G)$.
As noted by Fox, Luo and Pham~\cite{fox2024random}, it seems likely that $H=H_0$ does not satisfy \eqref{eq:balanceddegrees} with high probability if $d=\omega(n/\log n)$. To reach our desired subgraph $H$, then, we will modify (a typical instance of) $H_0$ by removing edges and adding other edges from $G\setminus H_0$. Note that we will have few options either when removing edges from around low-degree vertices in $H_0$, or when adding edges from $G\setminus H_0$ around high-degree vertices in $H_0$. By a delicate choice in our degree correction function, we will avoid doing so almost entirely. 

Our proof will have three steps (as mentioned at the end of Section~\ref{sec:intro}):

\begin{enumerate}
\item Take $H_0\subset G$, a typical irregular random subgraph with certain useful properties.
\item Find a desirable degree correction function $\psi:V(G)\to \mathbb{Z}$. 
\item Alter $H_0$ by adding/removing edges while controlling over- and under-corrections to degrees.
\end{enumerate}

\smallskip 

\noindent\textbf{Step 1: properties of $H_0$.} We will start with the following known properties of the irregular random subgraph. 
\begin{theorem}\label{thm:propsirregrandom}
Let $G$ be a $d$-regular graph on $n$ vertices, and let $H_0=H_0(G)$ be the irregular random subgraph. For each $0\leq k\leq d$, let $X_k$ be the number of vertices with degree $k$ in $H_0$. Then, for each $0\leq k\leq d$, $\mathbb{E} X_k=n/(d+1)$ and $\mathrm{Var}(X_k)\leq 17n/(d+1)$.
\end{theorem}
The value of the expectation in Theorem~\ref{thm:propsirregrandom} follows from a result of Frieze, Gould,
Karo\'nski and Pfender~\cite{frieze2002graph}, while its direct statement has the following elegant proof of Alon and Wei~\cite[Lemma~2.2]{alon2023irregular}. For each $v\in V(G)$ and $k\in [0,d]$, observe that (assuming, as holds with probability 1, that the numbers $x_u$, $u\in V(G)$, are all distinct) $v$ has degree $k$ in $H_0$ exactly if among the $d+1$ numbers $1-x_v$ and $x_u$, $u\in N_G(v)$, the number $1-x_v$ is the $(k+1)$st largest. As $1-x_v$ has the same distribution as each $x_u$, $u\in N_G(v)$, and they are all independent, we thus have $\P(d_{H_0}(v)=k)=1/(d+1)$, and the value of the expectation in Theorem~\ref{thm:propsirregrandom}  follows by linearity. The bound on the variance in Theorem~\ref{thm:propsirregrandom} was calculated by Fox, Luo and Pham~\cite[Theorem~2.1]{fox2024random}.

From Theorem~\ref{thm:propsirregrandom} and Chebyshev's inequality (Lemma~\ref{lem:chebyshev}),
the theorem already follows in the small-degree range; in particular,
we may assume throughout the main part of the proof that
$d=\omega(\log^{25} n)$.  In fact, the same argument gives the result
for \(d=o(n^{1/3})\) (see Section~\ref{sec:finalmainproof}).
For larger values of $d$ (in particular $d=\omega(n/\log n)$), we cannot expect to bound well the number of vertices with degree $i$ for each $0\leq i\leq d$. As vertices in $H_0$ with high- or low-degree play a particularly delicate role when creating our degree correction function, we focus on bounding their number well. We will need that there are likely plenty of vertices with degree at most $i$ and plenty with degree at least $d-i$, for each $0\leq i\leq d$ (later given as \ref{prop:smallorbigpots}). For other potential degrees, we (essentially) need there roughly to be the right number of vertices in $H_0$ with degree in intervals $[i,i+\ell]$, for each $0\leq i\leq d-\ell$ and $\ell\geq d^{2/3}/3$ (later given as \ref{prop:intervalsofpots}). These two properties are shown to hold with probability at least $1/2$ in $H_0$ in Section~\ref{sec:propsofrandomirreg}.

\smallskip

\noindent\textbf{Step 2: the degree correction function $\psi$.} For any given vertex $v\in V(G)$, conditioned on the value of $x_v$, we expect $v$ to have degree around $x_vd$ in $H_0$. Via a Chernoff bound (Lemma~\ref{lem:chernoff}) and a union bound, it follows that, with high probability, $|d_{H_0}(v)-x_vd|\leq d^{2/3}$ for each $v\in V(G)$. Similarly using a Chernoff bound and a union bound, with high probability there will be $(1 \pm \varepsilon)\ell \frac{n}{d+1}$ vertices $v\in V(G)$ with $\frac{i}{d+1}\leq x_v\leq \frac{i+\ell}{d+1}$ for any fixed $\varepsilon>0$ and all $0 \leq i \leq d-\ell$ and $\ell \geq d^{2/3}/3$. In combination, these properties make it straightforward to show that there is some function $\psi:V(G)\to \mathbb{Z}$ with $\|\psi\|=O(d^{2/3})$ such that $\#\{v\in V(G):d_{H_0}(v)-\psi(v)=k\}=(1+o(1))\frac{n}{d+1}$ for each $k\in [0,d]$.

Our challenge is that we want such a function with $\psi(v)<0$ when $d_{H_0}(v)$ is small and $\psi(v)>0$ when $d_{H_0}(v)$ is large, so that we are not too limited when making the degree corrections. Obtaining this property, and doing so while using relatively simple properties of the irregular random subgraph, motivates our construction of $\psi$. We sketch this construction at the start of Section~\ref{sec:degreecorrection}, before then carrying it out for Lemma~\ref{lem:randomirregularandcorrections}. The precise properties we obtain for $(H_0,\psi)$ are given later as \ref{prop:corr:smallcorrs}--\ref{prop:corr:options2} in that lemma. 

\smallskip

\noindent\textbf{Step 3: correcting degrees while managing unwanted changes.} 
Suppose now we have the subgraph $H_0$ with our desired function $\psi$. For each $i\in [0,d]$, we then let $W_i$ be the set of vertices $v\in V(G)$ for which $d_{H_0}(v)-\psi(v)=i$, i.e., the `pot' of vertices we aim to have degree $i$ in $H$. We will make our desired changes by (separately) removing $\psi(v)$ edges from $H_0$ around each vertex $v$ if $\psi(v)>0$, and adding $-\psi(v)$ edges from $G\setminus H_0$ around each vertex $v$ if $\psi(v)<0$. This causes some `unwanted changes' to the degrees of vertices on the non-$v$ side of each of these removed/added edges. To track the `unwanted changes', we will take an orientation of $G$ uniformly at random from all such orientations, to get the oriented graph $\vec{G}$. Letting $\vec{H}_0$ and $\vec{H}$ be the induced orientation of $H_0$ and $H$, respectively, to get $\vec{H}$ from $\vec{H}_0$, we will remove edges from $\vec{H}_0$ and add edges from $\vec{G}\setminus \vec{H}_0$ so that the change in out-degree at each vertex $v$ is equal to $-\psi(v)$. Therefore, the `unwanted changes' correspond to the change in in-degree at each vertex.

We will control these `unwanted changes' by choosing a small random subset $Z\subset V(G)$ and obeying the following two rules.
\begin{enumerate}[label = \roman{enumi}),itemsep=0pt, topsep=2pt]
\item Only edges directed into $Z$ are removed from $\vec{H}_0$ or added from $\vec{G}\setminus \vec{H}_0$.\label{rule:1}
\item The overall change in in-degree at each vertex is in $I-I$, where $I=\{0,1,2,$ $4,\dots,2^t\}$ with $t=\lfloor\log_2d\rfloor$.\label{rule:2}
\end{enumerate}
If $d=o(n/\log^6 n)$, then picking $Z\subset V(G)$ by including each vertex independently at random with probability $\eps/t^5$ (for any fixed $\eps>0$) will, with high probability, imply that if the changes can be made while obeying \ref{rule:1} and \ref{rule:2} then we will (essentially) get the subgraph we want. Indeed, rule \ref{rule:1} will ensure that (with high probability) we make our desired correction exactly for enough vertices that the lower bound in \eqref{eq:balanceddegrees} holds with $2\eps$ replacing $\eps$; i.e., for each $i\in [0,d]$, all of the vertices in $W_i\setminus Z$ will end up with degree $i$ in $H$ as they have in-degree 0 in the removed/added edges, and $W_i\cap Z$ is likely to be a small set.
Rule~\ref{rule:2} will ensure that (with high probability) not too many vertices will end up with any one fixed degree $i$ due to the unwanted changes. Indeed, note that any vertex with degree $i$ in the final graph $H$ which is not in $W_i$ must be in some $W_j\cap Z$ with $j-i\in I-I$. As $|I-I|\leq (t+2)^2$, this is likely to be few (at most $\eps n/(d+1)$) vertices.
For larger $d$, we will need only a slight alteration to the choice of the random set $Z$ (as described at the start of Section~\ref{sec:finalmainproof}).

In order to make the changes we want while obeying rule \ref{rule:2}, we first use an iterative random process to choose a set $E$ of edges of $\vec{H}_0$ which obey rule \ref{rule:1} and are such that each vertex $v$ with $\psi(v)>0$ has out-degree $\psi(v)$ in these edges and each vertex in $Z$ has in-degree in $I$ in these edges. At each step, we take a vertex $z\in Z$ that has not yet been considered, take the set of edges in $\vec{H}_0$ directed to $z$ from vertices which need more out-edges in $E$, and, if this set of edges is non-empty, then select from it $2^s\in I$ edges to add to $E$ uniformly at random for as large an integer value of $s$ as possible. Repeating this process for each vertex in $Z$ but with the edges in $\vec{G}\setminus \vec{H}_0$ and the vertices $v$ with $\psi(v)<0$, we can choose a set of edges $E'\subset \vec{G}\setminus \vec{H}_0$ such that, if the edges in $E$ are removed from $H_0$ and the edges in $E'$ are added, then \ref{rule:2} holds. 
In the first random process, the choice of $s$ ensures that each vertex still in want of out-edges in $E$ will get such an edge with probability at least $1/2$. Therefore, as each vertex $v$ with $\psi(v)>0$ will have plenty of out-neighbours in $Z$ in $\vec{H}_0$, our process is likely to assign it enough out-neighbours in $E$ to make the desired correction. Similarly, the second process is likely to choose the right number of out-edges of vertices $v$ with $\psi(v)<0$ for $E'$. These random processes are carried out in Section~\ref{sec:finalmainproof}, which completes our proof of Theorem~\ref{thm:AWasympt}.

\subsection{Notation} 
A graph \(G\) has vertex set \(V(G)\) and edge set \(E(G)\). For each \(v\in V(G)\), \(N_G(v)\) is the neighbourhood of \(v\) in
\(G\), and the degree of $v$ is $d_G(v)=|N_G(v)|$. Where \(U\subset V(G)\), we set $N_G(v,U)=N_G(v)\cap U$ and $d_G(v,U)=|N_G(v,U)|$.
For any subgraph \(H\subset G\), \(G\setminus H\) is the spanning
subgraph of \(G\) with vertex set \(V(G)\) and edge set
\(E(G)\setminus E(H)\). For graphs \(H_1\) and
\(H_2\), \(H_1+H_2\)
is the graph with vertex set $V(H_1)\cup V(H_2)$ and edge set \(E(H_1)\cup E(H_2)\).
If \(\vec G\) is an orientation of \(G\) and $v\in V(\vec{G})$, then \(N_{\vec G}^+(v)\) and
\(N_{\vec G}^-(v)\) denote the out- and in-neighbourhood of
\(v\), respectively, and $d_{\vec G}^+(v)=|N_{\vec G}^+(v)|$ and $d_{\vec G}^-(v)=|N_{\vec G}^-(v)|$.
Notation like \(N_{\vec G}^+(v,U)\) and $\vec{G}\setminus \vec{H}$ is used analogously to our notation for undirected graphs. If \(H\subset G\), then \(\vec H\) denotes the orientation
of \(H\) induced by \(\vec G\).

For a finite set \(S\), we sometimes use \(\#S=|S|\) to denote its cardinality. If \(a,b\in \mathbb{Z}\) with $b\geq a$, then
$[a,b]=\{x\in \mathbb{Z}:a\leq x\leq b\}$. For any $d\in \mathbb{N}$ with $d\geq 2$, $[0,d]=\{0,1,\dots,d\}$ and $[d]=[1,d]$, while we use $[0,1]$ as usual for the interval of real numbers.
For sets \(A,B\subset \mathbb{Z}\), we write $A+B=\{a+b:a\in A,\ b\in B\}$ and $A-B=\{a-b:a\in A,\ b\in B\}$, interpreting multiple sums and differences such as $I+I-I-I$ in the same way.
For real numbers \(a,b\) and \(c\ge 0\), $a=b\pm c$ means that \(b-c\le a\le b+c\) while $a=(1\pm c)b$ means that \((1-c)b\le a\le (1+c)b\). If
\(\psi:V(G)\to \mathbb{R}\), then $\|\psi\|=\max_{v\in V(G)}|\psi(v)|$. All logarithms are natural unless a base is explicitly
specified, as in \(\log_2\).

All asymptotic notation is with respect to \(n\to\infty\). We say that
an event holds with high probability if its probability tends to \(1\) as
\(n\to\infty\). We use standard hierarchy notation for constants. For example, an
assertion made under the assumption $\lambda \ll \varepsilon$ means that, for every fixed \(\varepsilon>0\), there is a
\(\lambda_0=\lambda_0(\varepsilon)>0\) such that the assertion holds for
all \(0<\lambda\le \lambda_0\), provided \(n\) is sufficiently large as a
function of \(\lambda\) and \(\varepsilon\). Longer hierarchies are
interpreted similarly, from right to left. Where rounding choices do not affect the estimates, we omit rounding symbols.


\subsection{Concentration inequalities} 
We will use the following standard form of Chebyshev's inequality (see, for example, \cite[Theorem~4.1.1]{alon2016probabilistic}).
\begin{lemma}[Chebyshev's inequality]\label{lem:chebyshev}
For any random variable $X$ with finite, non-zero variance, and any $k>0$,
$$\P\Big(|X-\mathbb EX|\ge k \sqrt{\mathrm{Var}(X)}\Big)\le \frac{1}{k^2}.$$
\end{lemma}
We will also use the following (see, for example, \cite[Corollary~2.2 and Theorem~2.10]{janson2011random}).
\begin{lemma}[Chernoff's bound]\label{lem:chernoff}
Let $X$ be a random variable which is binomially distributed with mean $\mu$.
Then, for any $0<\gamma \leq 1$, \[\mathbb{P}(|X-\mu|\geq \gamma \mu)\leq 2\exp\left({-\frac{\gamma^2\mu}{3}}\right).\]
\end{lemma}
We will also use the following concentration result for Lipschitz functions (see, for example, Chapter~7 of \cite{alon2016probabilistic}).
\begin{lemma}\label{lem:mcd}
Suppose that $X:\prod_{i=1}^N\Omega_i\to \mathbb{R}$ is $k$-Lipschitz. Then, for each $t>0$,
\[
\mathbb{P}(|X-\mathbb{E} X|>t)\leq 2\exp\left(\frac{-t^2}{2k^2N}\right).
\]
\end{lemma}

We will also need  a version of Azuma's inequality for submartingales. The following version is the neatest statement for our application. It is equivalent to \cite[Lemma~2.5]{kapralov2018optimal} and a similar statement can also be deduced easily from any standard  version of Azuma's inequality for submartingales (for example from \cite[Lemma~4.2]{wormald1999differential}).
\begin{lemma}[\cite{kapralov2018optimal}]\label{lem:azuma}
Let $p$ satisfy $0<p<1$. Let $(X_i)_{i =1}^n$ be Bernoulli random variables satisfying $\mathbb E(X_i|X_1, \dots, X_{i-1})\ge p$ for each $i\in [n]$. Then, for any $t>0$,
\[
\mathbb{P}\bigg(\sum_{i=1}^nX_i\leq pn- t \bigg) \leq  \exp \left( -\frac{t^2}{2(1-p)n+2t} \right).
\]
\end{lemma}


\section{Proof of Theorem~\ref{thm:AWasympt}}\label{sec:mainproof}

In Section~\ref{sec:propsofrandomirreg} we prove some properties of the irregular random subgraph (\ref{prop:smallorbigpots} and \ref{prop:intervalsofpots} in Lemma~\ref{lem:propsofrandomirregular}) that hold with probability at least $1/2$, and then use them to find a degree correction function in Section~\ref{sec:degreecorrection}, proving Lemma~\ref{lem:randomirregularandcorrections}. Finally, we prove Theorem~\ref{thm:AWasympt} in Section~\ref{sec:finalmainproof}.


\subsection{Properties of the irregular random subgraph}\label{sec:propsofrandomirreg}
We now prove the likely properties we need from the irregular random subgraph (\ref{prop:smallorbigpots} and  \ref{prop:intervalsofpots} below) in order to construct our degree correction function.

\begin{lemma}\label{lem:propsofrandomirregular}
Let $1/n\ll \lambda\ll \eps\leq 1$ and $d\leq \lambda n$ with $d=\omega(\log^{5} n)$. Let $G$ be a $d$-regular $n$-vertex graph. For each $v\in V(G)$, choose $x_v\in [0,1]$ independently and uniformly at random. Let $H_0$ be the graph with vertex set $V(G)$ and edge set $\{uv\in E(G):x_u+x_v\geq 1\}$. For each $i\in [0,d]$, let $V_i=\{v\in V(G):d_{H_0}(v)=i\}$. Then, with probability at least $1/2$, we have the following properties.
\begin{enumerate}[label = \emph{\textbf{A\arabic{enumi}}}]
\item For each $0\leq i\leq d$, $\sum_{j=0}^i|V_j|\geq (1-\frac{\eps}{4})\frac{(i+1)n}{d+1}$ and $\sum_{j=0}^i|V_{d-j}|\geq (1- \frac{\eps}{4})\frac{(i+1)n}{d+1}$.\labelinthm{prop:smallorbigpots}
\item\labelinthm{prop:intervalsofpots} For each $0\leq i\leq d$ and $\ell \ge d^{2/3}/3$ with $i+\ell\leq d$, we have $\sum_{j=i}^{i+\ell}|V_j|=(1\pm \frac{\eps}{4})\frac{(\ell+1)n}{d+1}$.
\end{enumerate}
\end{lemma}
\begin{proof}We will start by showing that \ref{prop:smallorbigpots} holds with probability $\ge 3/4$. Note that for this it is sufficient to show that ``For each $0\leq i\leq d$, $\sum_{j=0}^i|V_j|\geq (1-\frac{\eps}{4})\frac{(i+1)n}{d+1}$'' holds with probability $\ge 7/8$. Indeed, if this were true, then by symmetry between $H_0$ and its complement in $G$ we would also have that ``For each $0\leq i\leq d$, $\sum_{j=0}^i|V_{d-j}|\geq (1-\frac{\eps}{4})\frac{(i+1)n}{d+1}$'' with probability $\ge 7/8$ --- and hence get that \ref{prop:smallorbigpots} holds with probability $\ge 3/4$ by a union bound.

Pick, then, $C$ with $\lambda\ll 1/C\ll \eps$, and note that we can assume that $\eps\ll 1$ as the property in the lemma is stronger when $\eps$ is smaller. 
By Theorem~\ref{thm:propsirregrandom} and Chebyshev's inequality (Lemma~\ref{lem:chebyshev}) with $k=\frac{\eps}{4}\sqrt{\frac { n}{17(d+1)}}$ we have that, for each $0\leq j\leq d$, 
\begin{equation}\label{eq:whendissmall}
    \P\left(|V_j|\ne \left(1\pm\frac{\eps}{4}\right)\frac n{d+1}\right)= \P\left(\Big||V_j|-\frac n{d+1}\Big|> \frac{\eps}{4}\sqrt{\frac { n}{17(d+1)}}\sqrt{\frac {17n}{d+1}}\right)\le \frac{272(d+1)}{\eps^2n}\le  \frac{1}{40C},
    \end{equation}
    where we have used that $d\leq\lambda n$ and $\lambda\ll 1/C,\eps$.
By a union bound, with probability $\ge 19/20$, we thus have that $|V_j|= \big(1\pm\frac{\eps}{4}\big)\frac n{d+1}$ for each $0\leq j\le C$. This implies that the first part of \ref{prop:smallorbigpots} holds for each $0\leq i\leq C$ with probability $\geq 19/20$. For the rest of the range,  we first prove the following.
\begin{claim}\label{claim:thisone}
Let $C\leq i\leq d$. Then, with probability $\ge 1-40\eps^{-1}\exp(-\frac{\eps^2i}{10^4})$, the following hold.
\begin{enumerate}[label ={\textbf{B\arabic{enumi}}}]
\item There are at least $(1-\frac{\eps}{8})\frac{in}{d+1}$ vertices $v\in V(G)$ with $x_v\leq (1- \frac{\eps}{16})\frac{i}{d+1}$.\label{prop:extremepots:part1}
\item There are at most $\frac{\eps i n}{16(d+1)}$ vertices $v\in V(G)$ with $x_v\leq (1-\frac{\eps}{16})\frac{i}{d+1}$ for which there are at least $i$ vertices $u\in N_G(v)$ with $x_u\geq 1-(1-\frac{\eps}{16})\frac{i}{d+1}$.\label{prop:extremepots:part2}
\end{enumerate}
\end{claim}
\claimproofstart[Proof of Claim~\ref{claim:thisone}]
Each $x_v$ is chosen independently, so the number of vertices $v\in V(G)$ with $x_v\le (1-\frac{\eps}{16})\frac{i}{d+1}$ is binomially distributed with mean $(1-\frac{\eps}{16})\frac{in}{d+1}$. By Lemma~\ref{lem:chernoff} with $\gamma=\frac{\eps/16}{1-\eps/16}\geq \frac{\eps}{17}$,  \ref{prop:extremepots:part1} thus fails with probability $\le2\exp\big(-\frac{1}{3}\left(\frac{\eps}{17}\right)^2(1-\frac{\eps}{16})\frac{in}{(d+1)}\big)\le 2\exp(-\frac{\eps^2i}{10^4})\le 2\eps^{-1}\exp(-\frac{\eps^2i}{10^4})$. Similarly, for each $v\in V(G)$, the number of vertices $u\in N_G(v)$ with $x_u\geq 1-(1-\frac{\eps}{16})\frac{i}{d+1}$ is binomially distributed with mean $(1-\frac{\eps}{16})\frac{id}{d+1}$. 
Thus, using Lemma~\ref{lem:chernoff} and independence, the probability that ``$x_v\le (1-\frac{\eps}{16})\frac{i}{d+1}$'' and ``there are at least $i$ vertices $u\in N_G(v)$ with $x_u\geq 1-(1-\frac{\eps}{16})\frac{i}{d+1}$'' is at most 
$\frac{i}{d+1}\cdot 2\exp\big(-\frac{1}{3}(\frac{\eps}{17})^2 (1-\frac{\eps}{4})\frac{id}{d+1}\big)\le \frac{i}{d+1}\cdot 2\exp(-\frac{\eps^2 i}{10^4})$.
Thus the expected number of such $v\in V(G)$ is at most $\frac{2in}{d+1}\exp(-\frac{\eps^2i}{10^4})$. Hence, by Markov's inequality, the probability that \ref{prop:extremepots:part2} does not hold is at most $\big(\frac{2in}{d+1}\exp(-\frac{\eps^2i}{10^4})\big)/\big(\frac{\eps i n }{16(d+1)}\big)=32\eps^{-1}\exp(-\frac{\eps^2i}{10^4})$. Thus, the probability that either \ref{prop:extremepots:part1} or \ref{prop:extremepots:part2} does not hold is at most $40\eps^{-1}\exp(-\frac{\eps^2i}{10^4})$, as required.
\claimproofend
By Claim~\ref{claim:thisone} and a union bound, we have that
 \ref{prop:extremepots:part1} and \ref{prop:extremepots:part2} hold for all $i=C,C+1, \dots, d$ with probability $\ge1-\sum_{i=C}^{\infty}40\eps^{-1}\exp(-\frac{\eps^2i}{10^4})=1-\frac{40\eps^{-1}\exp(-\eps^2C/10^4)}{1-\exp(-\eps^2/10^4)}\ge 1-8\cdot 10^5\eps^{-3}\exp(-\frac{\eps^2C}{10^4})\geq 19/20$. Assume, then, that \ref{prop:extremepots:part1} and \ref{prop:extremepots:part2} hold for all $i=C,C+1, \dots, d$. Let $C\leq i\leq d$. Call a vertex $v\in V(G)$ \emph{$i$-good} if $x_v\le (1-\frac{\eps}{16} )\frac{i}{d+1}$ and there are at most $i$ vertices $u\in N_G(v)$ with $x_u\ge 1- (1-\frac{\eps}{16} )\frac{i}{d+1}$. Using that $(1-\frac{\eps}{16} )\frac{i}{d+1}+ (1- (1-\frac{\eps}{16} )\frac{i}{d+1})=1$  we have that $i$-good vertices have degree $\le i$ in $H_0$.
Thus if \ref{prop:extremepots:part1} and \ref{prop:extremepots:part2} hold for some $C\leq i\leq d$, then $\sum_{j=0}^i|V_j|$ is at least the number of $i$-good vertices,  which is at least $(1-\frac{\eps}{8})\frac{in}{d+1}-\frac{\eps i n}{16(d+1)}\geq (1-\frac{\eps}{4})\frac{(i+1)n}{d+1}$ as $i\geq C$ and $1/C\ll \eps$. Thus, with probability $\geq 19/20$, the first part of \ref{prop:smallorbigpots} holds when $C\leq i\leq d$. Therefore, with probability $\geq 7/8$, the first part of \ref{prop:smallorbigpots} holds. 

\smallskip

Thus, it suffices to prove that \ref{prop:intervalsofpots} holds with probability $\geq 3/4$. 
For each $v\in V(G)$, condition on $x_v$ and let
$Y_v=|\{u\in N_G(v):x_v+x_u\geq 1\}|$. Then, $Y_v$ is binomially distributed with mean $\mu:= x_v d$. If $d^{3/5} \leq \mu \leq d$, then Lemma~\ref{lem:chernoff}, applied with
$\gamma=d^{3/5}/\mu$, gives
\[
\P(|Y_v-\mu|>d^{3/5})\le
2\exp\left(-\frac{d^{6/5}}{3\mu}\right)
\le 2\exp\left(-\frac{d^{1/5}}{3}\right).
\]
Otherwise, we have $\mu<d^{3/5}$ and as the lower-tail event is empty it
suffices to bound $\P(Y_v>\mu+d^{3/5})$. For this, let $Y'_v$ be a binomial random variable with distribution
    $\operatorname{Bin}(d,\mu'/d)$, where
    $\mu':=(\mu+d^{3/5})/2\geq \mu$.  Since the random variable $Y_v$ is stochastically
dominated by $Y'_v$, using Lemma~\ref{lem:chernoff} with
$\gamma=1$, we get
\[
\P(Y_v>\mu+d^{3/5})
\leq
\P\left(Y'_v>\mu+d^{3/5}\right)
\leq
\P\left(Y'_v>2 \mu'\right)
\leq
2\exp\left(-\frac{\mu+d^{3/5}}{6}\right).
\]
Thus, for all $v\in V(G)$, $
\P(|d_{H_0}(v)-x_v  d|>d^{3/5})
\leq 2\exp\left(-\frac{d^{1/5}}{6}\right)$.
Since $d=\omega(\log^{5} n)$,  with probability $\ge 1-1/8$ we have
\begin{enumerate}[label ={\textbf{B\arabic{enumi}}}]\addtocounter{enumi}{2}
\item For each $v\in V(G)$, $|d_{H_0}(v)-x_v  d|\leq d^{3/5}$.\label{prop:intervalsofpots:part1}
\end{enumerate}
Furthermore, using Lemma~\ref{lem:chernoff}, for any $a,b\in [0,1]$ with $b\ge a+d^{-2/5}$ the following holds with probability at least $1-2\exp(-\frac{1}{3}(\frac{\eps}{8})^2d^{-2/5}n)\ge 1-(16d(d+1))^{-1}$.
\begin{enumerate}[label ={\textbf{B\arabic{enumi}}}]\addtocounter{enumi}{3}
\item  $|\{v\in V(G):a\leq x_v  \leq b\}|=(1\pm \eps/8)(b-a)n$.\label{prop:intervalsofpots:part2}
\end{enumerate}
Thus, by a union bound over all $0\le i\le d$ and integer $\ell\ge d^{2/3}/3$ with $i+\ell \le d$ (noting there are at most $d(d+1)$ such choices) with probability $\ge 3/4$, we have that \ref{prop:intervalsofpots:part1} holds and that \ref{prop:intervalsofpots:part2} holds with $a=\frac{i}{d}+d^{-2/5}$, $b=\frac{i+\ell}{d}-d^{-2/5}$ and with $a=\max\{\frac{i}{d}-d^{-2/5},0\}$, $b=\min\{\frac{i+\ell}{d}+d^{-2/5},1\}$ for each $0\le i\le d$ and $\ell\ge d^{2/3}/3$ with $i+\ell \le d$. Assuming this, we will now show that \ref{prop:intervalsofpots} holds.

Consider, then, some $0\le i\le d$ and $\ell\ge d^{2/3}/3$ with $i+\ell \le d$. Let 
\[
X=\{v\in V(G): i+d^{3/5}\le x_v  d\le i+\ell-d^{3/5}\}\] and \[Y=\{v\in V(G): i-d^{3/5}\le x_v  d\le i+\ell+d^{3/5}\}.\] 
Note that, by \ref{prop:intervalsofpots:part1}, we have $X\subset \bigcup_{j=i}^{i+\ell} V_j\subset Y$. 
Thus, by \ref{prop:intervalsofpots:part2} applied with $a=\frac{i}{d}+d^{-2/5}$ and $b=\frac{i+\ell}{d}-d^{-2/5}$ (so that $b-a=\ell/d-2d^{-2/5}\geq d^{-2/5}$), we have 
\[
|X|\geq \left(1-\frac{\eps}{8}\right)\cdot \frac{\ell-2d^{3/5}}{d}\cdot n\geq \left(1-\frac{\eps}{4}\right)\frac{(\ell+1)}{d+1}\cdot n,
\]
where the last inequality holds as $\ell\geq d^{2/3}/3$ and $d\gg \eps^{-1}$. Similarly, by \ref{prop:intervalsofpots:part2} applied with $a=\max\{\frac{i}{d}-d^{-2/5},0\}$ and $b=\min\{\frac{i+\ell}{d}+d^{-2/5},1\}$,
\[
|Y|\leq \left(1+\frac{\eps}{8}\right)\cdot \frac{\ell+2d^{3/5}}{d}\cdot n\leq \left(1+\frac{\eps}{4}\right)\frac{(\ell+1)}{d+1}\cdot n,
\]
and thus \ref{prop:intervalsofpots} holds.
\end{proof}


\subsection{A degree correction function for the random irregular subgraph}\label{sec:degreecorrection}
To form our correction function using \ref{prop:smallorbigpots} and \ref{prop:intervalsofpots}, for some appropriate $r$ we will first partition $[0,d]$ into intervals $I_0,I_1,\dots,I_{2r+1}$ in increasing order, such that all but the first and last intervals have the same size, $d^{2/3}/2$. We then assign $\psi$ for most of the vertices $v\in V(G)$ with $d_{H_0}(v)\in I_0\cup I_1\cup\dots \cup I_r$ in order to guarantee that, for each $i\in I_0\cup I_1\cup\dots \cup I_r$, $d_{H_0}(v)-\psi(v)=i$ for at least $(1-\eps)n/(d+1)$ such vertices $v\in V(G)$. Where $i\in I_j$ with $j\in [r]$, we will do this using vertices $v$ with $d_{H_0}(v)\in I_{j-1}$, using \ref{prop:intervalsofpots}. This ensures that each $|\psi(v)|$ is not too large for any such assignment, and $\psi(v)\leq 0$ for each of these vertices $v$. Where $i\in I_0$ (in addition to the case $i\in I_1$), we will also do this using vertices $v$ with $d_{H_0}(v)\in I_{0}$, using that $I_0$ will be a larger interval than $I_1$. With more care, using \ref{prop:smallorbigpots}, we make sure that all of these values of $\psi(v)$ are also at most 0. Then, for the vertices $v$ with $d_{H_0}(v)\in I_0\cup I_1\cup\dots\cup I_r$ for which $\psi(v)$ has not yet been chosen, we assign $\psi$-values to `spread out' the values of $d_{H_0}(v)-\psi(v)$ over these vertices so that we do not affect too much the nice distribution we have created, and so that all these values are between $-d^{3/4}$ and $0$. 
This will allow us to choose $\psi$-values for all vertices $v$ with $d_{H_0}(v)\in I_0\cup I_1\cup\dots\cup I_r$ with the properties we want (see \ref{prop:psi:1} and \ref{prop:psi:2} later). We then similarly assign $\psi$-values for all vertices $v$ with $d_{H_0}(v)\in I_{r+1}\cup I_{r+2}\cup\dots\cup I_{2r+1}$, and show that together this gives us the desired function $\psi$. We now carry this all out in detail, in order to prove the following result.

\begin{lemma}\label{lem:randomirregularandcorrections}
Let $1/n\ll \lambda\ll \eps\leq 1$ and let $d\leq \lambda n$ satisfy $d=\omega(\log^5 n)$. Let $G$ be a $d$-regular $n$-vertex graph. Then, there is a spanning subgraph $H_0\subset G$ and a function $\psi:V(G)\to \mathbb{Z}$ such that the following hold.
\begin{enumerate}[label = {\textbf{\emph{C\arabic{enumi}}}}]
\item For each $v\in V(G)$, $|\psi(v)|\leq d^{3/4}$.\labelinthm{prop:corr:smallcorrs}
\item  For each $i\in [0,d]$, $|\{v\in V(G):d_{H_0}(v)-\psi(v)=i\}|=(1\pm \eps)\frac{n}{d+1}$.\labelinthm{prop:corr:goodpots}
\item For each $v\in V(G)$, if $\psi(v)> 0$ then $v$ has at least $d/2$ neighbours in $H_0$.\labelinthm{prop:corr:options1}
\item For each $v\in V(G)$, if $\psi(v)< 0$ then $v$ has at least $d/2$ neighbours in $G\setminus H_0$.\labelinthm{prop:corr:options2} 
\end{enumerate}
\end{lemma}
\begin{proof} Let $C$ be such that $\lambda \ll 1/C\ll \eps$. Using Lemma~\ref{lem:propsofrandomirregular}, let $H_0\subset G$ be such that \ref{prop:smallorbigpots} and \ref{prop:intervalsofpots} hold. Let $n_{\min}=(1-\frac{\eps}{3})\frac{n}{d+1}$ and $n_{\max}=(1+\eps)\frac{n}{d+1}$. We will choose $\psi:V(G)\to \mathbb{Z}$ with $n_{\min}\le |\{v\in V(G):d_{H_0}(v)-\psi(v)=i\}|\le n_{\max}$ for each $i\in [0,d]$.

Let $\ell=d^{2/3}/2$.  Find an integer $r$ such that we can divide $[0,d]=I_0\cup I_1\cup \dots \cup I_{2r}\cup I_{2r+1}$ so that \textbf{a)} $|I_j|=\ell$ for each $j\in [2r]$, \textbf{b)} $C\ell\leq |I_0|,|I_{2r+1}|\leq 2C\ell$, \textbf{c)} $\lfloor d/2 \rfloor\in I_r$ and $\lfloor d/2 \rfloor+1\in I_{r+1}$, and \textbf{d)} for each $j\in [0,2r]$, $\max I_j<\min I_{j+1}$. 

Let $V_i=\{v\in V(G):d_{H_0}(v)=i\}$ for each $i\in [0,d]$.
For each $j\in [0,2r+1]$, let $Z_j=\cup_{i\in I_j}V_i$. We will define a function $\psi:V(G)\to \mathbb{Z}$ such that, if $v\in Z_0\cup Z_1\cup\dots \cup Z_r$ then $\psi(v)\leq 0$ and if $v\in Z_{r+1}\cup Z_{r+2}\cup\dots\cup  Z_{2r+1}$ then $\psi(v)\geq 0$. Thus, \ref{prop:corr:options1} and \ref{prop:corr:options2} will hold. We start by defining $\psi$ on $Z_0\cup Z_1\cup \dots \cup Z_r$, for the following claim.
\begin{claim}\label{clm:new}
    There exists $\psi:\bigcup_{j=0}^r Z_j\to \mathbb{Z}$ such that $-d^{3/4}\le \psi(v)\le 0$ for each $v\in \bigcup_{j=0}^r Z_j$ and the following hold.
    \begin{enumerate}[label = {\textbf{{D\arabic{enumi}}}}]
        \item For all $i\in \bigcup_{j=0}^r I_j$, we have  $n_{\min}\le |\{v\in \bigcup_{j=0}^r Z_j:d_{H_0}(v)-\psi(v)=i\}|\le n_{\min}+\left(\frac{2\eps}{3}\right) \frac{n}{d+1}$.\label{prop:psi:1}
        \item For all $i\in \bigcup_{j=r+1}^{2r+1} I_j$, we have  $|\{v\in \bigcup_{j=0}^r Z_j:d_{H_0}(v)-\psi(v)=i\}|\le \left(\frac{2\eps}{3}\right) \frac{n}{d+1}$.\label{prop:psi:2}
    \end{enumerate}
\end{claim}
\claimproofstart[Proof of Claim~\ref{clm:new}] Order $V(G)=\{v_1, \dots, v_n\}$ with $d_{H_0}(v_1)\le \dots \le d_{H_0}(v_n)$. For each $i\in I_0\cup I_1$, let $U_i=\{v_{n_{\mathrm{min}}i+1}, \dots, v_{n_{\mathrm{min}}(i+1)}\}$. For any $t\in I_0\cup I_1$, since $\bigcup_{i\in [0,t]}V_i=\{v_s: d_{H_0}(v_s)\le t\}$, by \ref{prop:smallorbigpots} we have $|\bigcup_{i\in [0,t]}V_i|\ge (t+1)\cdot n_{\mathrm{min}}=|\bigcup_{i\in [0,t]}U_i|$, and, hence, $\bigcup_{i\in [0,t]}U_i\subset \bigcup_{i\in [0,t]}V_i$. This implies that, for all $v\in U_i$ with $i\in I_0\cup I_1$, we have $d_{H_0}(v)\le i$, while, as $i\in I_0\cup I_1$, $\ell=d^{2/3}/2$, and $d=\omega(\log^5n)$, we also have that $d_{H_0}(v)-i\geq -i\ge - (2C+1)\ell\geq -d^{3/4}$. Furthermore, by \ref{prop:smallorbigpots}, and as $\ell/d=d^{-1/3}/2$ and $d=\omega(\log^5n)$,\\
\[
\Big|\bigcup_{i\in {I_0}}V_i\Big|\ge \left(1-\frac{\eps}{4}\right)\frac{|I_0|n}{d+1}\ge \left(1-\frac{\eps}{3}\right)\frac{(|I_0|+\ell)n}{d+1}=   (|I_0|+\ell) \cdot n_{\mathrm{min}}=\Big|\bigcup_{i\in I_0\cup I_1}U_i\Big|,
\]
and thus $\bigcup_{i\in I_0\cup I_1}U_i\subset \bigcup_{i\in I_0}V_i=Z_0$.

Then, using \ref{prop:intervalsofpots}, for each $2\leq j\leq r$, take disjoint sets $U_i$, $i\in I_{j}$, with size $n_{\mathrm{min}}$ each in $Z_{j-1}$. Note  that then, for each $2\leq j\leq r, i\in I_{j}, v\in U_i$, we have $d_{H_0}(v)\le \max I_{j-1}< \min I_j\le i$ and $d_{H_0}(v)\geq \min I_{j-1}\ge \min I_j-2\ell\geq i-d^{3/4}$ for each $v\in U_i$. To summarise, in total we have disjoint sets $U_i$, $i\in I_0\cup I_1\cup \ldots \cup I_r$, in $Z_0\cup Z_1\cup \ldots \cup Z_{r-1}$, each with size $n_{\mathrm{min}}$, such that $-d^{3/4}\leq d_{H_0}(v)-i\leq 0$ for each $v\in U_i$, $i\in I_0\cup\dots\cup I_r$.
Furthermore, for each $j\in [r-1]$, by \ref{prop:intervalsofpots}, we have 
\begin{equation}\label{eq:Zprime}
\Big|Z_{j}\setminus \bigcup_{i\in I_{j+1}}U_i\Big|\leq \left(1+\frac{\eps}{4}\right)\frac{\ell\cdot n}{d+1}-\ell \cdot n_{\mathrm{min}}= \left(\frac{7\eps}{12}\right)\frac{\ell\cdot n}{d+1}.
\end{equation}

For each $i\in I_0\cup I_1\cup\dots \cup I_r$ and $v\in U_i$, set $\psi(v)=d_{H_0}(v)-i$. For each $j\in [0,r]$, let $Z'_j=Z_j\setminus \bigcup_{i\in I_0\cup I_1\cup\dots \cup I_r}U_i$, so that, if $j\in [r-1]$, then $|Z'_j|\leq 7\eps \ell n/12(d+1)$ by \eqref{eq:Zprime}. Then, for each $j\in [0,r]$, pick some function $f_j: Z_j'\to I_{j+1}\cup I_{j+2}\cup \dots \cup I_{j+C^2}$ such that $|f_j^{-1}(i)|\in \{\lfloor |Z'_j|/C^2\ell\rfloor, \lceil |Z'_j|/C^2\ell\rceil\}$ for each $i\in I_{j+1}\cup I_{j+2}\cup \dots \cup I_{j+C^2}$ (using simply that each $I_{j'}$ with $j+1\leq j'\leq j+C^2$ has size $\ell$). For each $j\in [0,r]$ and $v\in Z_j'$, set $\psi(v)=d_{H_0}(v)-f_j(v)$, and note that, as $v\in V_i$ for some $i\in I_j$ and $\min I_{j+1}\leq f_j(v)\leq \max I_{j+C^2}$, we have that $0\geq \psi(v)\geq -(C^2+1)\ell\geq -d^{3/4}$.

Noting that we have chosen $\psi(v)$ for each $v\in \bigcup_{j=0}^{r}Z_j$, we now show that \ref{prop:psi:1} and \ref{prop:psi:2} hold. Firstly, for each $i\in \bigcup_{j=0}^rI_j$, we have that 
\[
\Big|\Big\{v\in \bigcup_{j=0}^r Z_j:d_{H_0}(v)-\psi(v)=i\Big\}\Big|\geq |U_i|= n_{\min},
\]
so that the lower bound in \ref{prop:psi:1} holds. Furthermore, for each $j'\in [0,r]$ and $i\in I_{j'}$, we have
\begin{align*}
\Big|\Big\{v\in \bigcup_{j=0}^r Z_j:d_{H_0}(v)-\psi(v)=i\Big\}\Big|&\leq |U_i|+\sum_{j=0}^r|f^{-1}_j(i)|\leq n_{\mathrm{min}}+\left\lceil\frac{|Z_0'|}{C^2\ell}\right\rceil+
\sum_{j=\max\{1,j'-C^2\}}^{j'-1}\!\!\left\lceil\frac{|Z_j'|}{C^2\ell}\right\rceil\\
&\!\!\!\!\!\overset{\ref{prop:intervalsofpots},\eqref{eq:Zprime}}{\leq} n_{\mathrm{min}}+\left\lceil\frac{1}{C}\cdot 2\frac{n}{(d+1)}\right\rceil
+C^2\cdot \left\lceil\left(\frac{7\eps}{12}\right)\frac{n}{C^2(d+1)}\right\rceil\\
&\leq n_{\mathrm{min}}+\left(\frac{2\eps}{3}\right) \frac{n}{d+1},
\end{align*}
where we have used that $1/C\ll \eps$. Therefore, the upper bound in \ref{prop:psi:1} holds, and thus \ref{prop:psi:1} holds.

Finally, for each $i\in \bigcup_{j=r+1}^{2r+1} I_j$, by construction, if $v\in \bigcup_{j=0}^r Z_j$ with $d_{H_0}(v)-\psi(v)=i$, then $v\in Z'_j$ for some $r-C^2<j\leq r$. Therefore, 

\begin{align*}
\Big|\Big\{v\in \bigcup_{j=0}^r Z_j:d_{H_0}(v)-\psi(v)=i\Big\}\Big|&\le \left\lceil\frac{{|Z_r|}}{C^2\ell}\right\rceil+\sum_{j=r-C^2+1}^{r-1}\!\!\left\lceil\frac{|Z_j'|}{C^2\ell}\right\rceil\\
&\!\!\!\!\!\overset{\ref{prop:intervalsofpots},\eqref{eq:Zprime}}{\leq} \frac{2}{C^2}\cdot \frac{n}{d+1}+C^2\cdot \left\lceil\left(\frac{7\eps}{12}\right)\frac{n}{C^2(d+1)}\right\rceil\leq \left(\frac{2\eps}{3}\right) \frac{n}{d+1},
\end{align*}
where we have again used that $1/C\ll \eps$. Thus, \ref{prop:psi:2} holds.
\claimproofend
Using a symmetric argument, we can define  $\psi$ on $Z_{r+1}\cup \dots\cup Z_{2r+1}$ such that $0\le \psi(v)\le d^{3/4}$ for each $v\in Z_{r+1}\cup \dots\cup Z_{2r+1}$ and the following hold.
\begin{enumerate}[label = {\textbf{{D\arabic{enumi}}}}]\addtocounter{enumi}{2}
 \item For all $i\in \bigcup_{j=r+1}^{2r+1} I_j$, we have $n_{\min}\le |\{v\in \bigcup_{j=r+1}^{2r+1} Z_j:d_{H_0}(v)-\psi(v)=i\}|\le n_{\min}+\left(\frac{2\eps}{3}\right) \frac{n}{d+1}$.\label{prop:psi:3}
 \item For all $i\in \bigcup_{j=0}^{r} I_j$, we have $|\{v\in \bigcup_{j=r+1}^{2r+1} Z_j:d_{H_0}(v)-\psi(v)=i\}|\le \left(\frac{2\eps}{3}\right) \frac{n}{d+1}$.\label{prop:psi:4}
\end{enumerate}

We can now observe that $\psi$ has the properties we require. Property \ref{prop:corr:smallcorrs}  is immediate from the allowed range for $\psi(v)$ above.
For each $i\in [0,d]$, using \ref{prop:psi:1} and \ref{prop:psi:4} if $i\in \bigcup_{j=0}^{r} I_j$ and using \ref{prop:psi:2} and \ref{prop:psi:3} if $i\in \bigcup_{j=r+1}^{2r+1} I_j$, we have 
$n_{\min}\le |\{v\in V(G):d_{H_0}(v)-\psi(v)=i\}|\le n_{\max}$, where we have used that $n_{\min}+2\cdot \left(\frac{2\eps}{3}\right) \frac{n}{d+1}=n_{\max}$. Thus, \ref{prop:corr:goodpots} holds. 
 Property \ref{prop:corr:options2} holds because the only vertices with $\psi (v)<0$ are in $Z_0\cup \dots \cup Z_r$, which are the vertices with degree in $H_0$ in $I_0\cup \dots \cup I_r=\{i:0\leq i\leq \lfloor d/2\rfloor\}$. Property \ref{prop:corr:options1} holds similarly.
\end{proof}


\subsection{Proof of Theorem~\ref{thm:AWasympt}}\label{sec:finalmainproof}

Finally, using the proof sketch in Section~\ref{sec:sketch} and the subgraph $H_0$ and function $\psi$ satisfying \ref{prop:corr:smallcorrs}--\ref{prop:corr:options2}, we will prove Theorem~\ref{thm:AWasympt}. The main detail we have not yet covered is the choice of the random set $Z$ when $d=\Omega(n/\log n)$. Due to this case, we pick $Z\subset V(G)$ randomly by including each element with probability $p=\eps/t^5$ (with $t=\lfloor \log_2d\rfloor$), but throw out any vertex $v\in Z$ if, in the random process described loosely in Section~\ref{sec:sketch}, the degree of $v$ could be moved by some $j\in I-I$, together with too many other vertices, to have the same degree in the final graph $H$ (see the condition at \eqref{eq:pickingZ}).  
Randomly orienting the edges of $G$ to get $\vec{G}$, we will then show (for Claim~\ref{clm:directededges} below) that vertices $v$ with $\psi(v)>0$ will have many out-edges into $Z$ in the induced orientation of $H_0$ (see~\ref{prop:outneighboursplus}) and a similar property for vertices $v$ with $\psi(v)<0$ (see~\ref{prop:outneighboursminus}). We then carry out our random process (see \ref{step:a}--\ref{step:d} below) to produce $H$ and show (for Claim~\ref{clm:propsofHrGr} below) that this process will make the corrections requested by $\psi$ with some additional unwanted changes to the degrees of vertices in $Z$ (given by the in-degrees of changed edges). Finally, from the choice of $Z$ and the properties of the random process (in particular~\ref{prop:final:2}) we conclude that $H$ is as desired.

\begin{proof}[Proof of Theorem~\ref{thm:AWasympt}] For notational convenience, we will prove the result with $2\eps$ in place of $\eps$, which is sufficient as $\eps>0$ is arbitrary. Let $n_0$ and $\lambda$ be such that $1/n_0\ll \lambda\ll \eps$. Note that it is sufficient to prove the result for each $n\geq n_0$, as the property in Theorem~\ref{thm:AWasympt} will then hold with $\lambda'=\min\{\lambda,1/n_0\}$ (with $2\eps$ in place of $\eps$).

Let, then, $n\geq n_0$ and $d\leq \lambda n$, and let $G$ be a $d$-regular, $n$-vertex graph. If $d=o(n^{1/3})$, then note that it follows from Theorem~\ref{thm:propsirregrandom} and Lemma~\ref{lem:chebyshev} (working similarly to \eqref{eq:whendissmall}) that, with high probability, there are $(1\pm 2\eps)n/(d+1)$ vertices with degree $i$ in the irregular random subgraph $H_0$ of $G$ for each $i\in [0,d]$, in which case we can take $H=H_0$. Thus, we can certainly assume that $d=\omega(\log^{25}n)$.

Using Lemma~\ref{lem:randomirregularandcorrections}, take a spanning subgraph $H_0\subset G$ and a function $\psi:V(G)\to \mathbb{Z}$ such that \ref{prop:corr:smallcorrs}--\ref{prop:corr:options2} hold. The rest of the proof proceeds as follows. First we choose a random set $Z\subset V(G)$ and prove properties  \ref{prop:WZsmall}--\ref{prop:outneighboursminus} hold with high probability. Then we use a random process to build graphs $H_r\subset H_0$, $G_r\subset G\setminus H_0$ and prove properties \ref{prop:final:1}--\ref{prop:final:3} hold with high probability. Finally, we set $H=H_r+(G_0\setminus G_r)$ and show that (a subset of) the properties \ref{prop:corr:smallcorrs}--\ref{prop:corr:options2}, \ref{prop:WZsmall}--\ref{prop:outneighboursminus}, \ref{prop:final:1}--\ref{prop:final:3} imply that \eqref{eq:balanceddegrees} holds. 

For each $i\in [0,d]$, let $W_i=\{v\in V(G):d_{H_0}(v)-\psi(v)=i\}$, to get the partition $V(G)=W_0\cup W_1\cup\dots \cup W_d$. For convenience of notation, for any $i\in \mathbb{Z}$ with $i<0$ or $i>d$, let $W_i=\emptyset$.
Let $t=\lfloor \log_2 d\rfloor$. Let $I=\{0,1,2,4,\dots,2^t\}$. Let $p=\eps/t^5$. For each $v\in V(G)$, choose $y_v\in [0,1]$ independently and uniformly at random.
 Let $J$ be the set of $i\in [0,d]$ such that
\begin{equation}\label{eq:pickingZ}
\Bigg|\bigg\{v\in \bigcup_{j\in I+I-I-I}W_{i-j}:y_v\leq p\bigg\}\Bigg|\leq \frac{\eps n}{d+1}.
\end{equation}
Let $Z$ be the set of vertices $v\in V(G)$ such that $y_v\leq p$ and $v\in W_i$ for some $i\in J$. Then, by the choice of $J$, we have the following property.

\begin{enumerate}[label = {\textbf{E\arabic{enumi}}}]
\item For each $i\in [0,d]$,\label{prop:WZsmall}
\[ 
\sum_{j\in I-I}\left|W_{i-j}\cap Z\right|\leq \frac{\eps n}{d+1}.
\]
\end{enumerate}
Indeed, if $W_{i-j}\cap Z\neq\emptyset$ for some $j \in I-I$, then, using the definition of $Z$, we have $k:=i-j\in J$. For every
$j' \in I-I$, we have $W_{i-j'}=W_{k+j-j'}=W_{k-(j'-j)}$,
and $j'-j\in (I-I)-(I-I)=I+I-I-I$. Hence,
\[
\bigcup_{j'\in I-I} W_{i-j'}
\;\mathrel{\mathlarger{\subset}}\;
\bigcup_{s\in I+I-I-I} W_{k-s}.
\]
Since $k\in J$, the latter set contains at most $\varepsilon n/(d+1)$ vertices $v$
with $y_v\le p$, giving \ref{prop:WZsmall}.

For each $uv\in E(G)$, independently orient $uv$ uniformly at random, calling the resulting oriented graph $\vec{G}$. From now on, for any $H\subset G$, we will denote the corresponding oriented graph by $\vec{H}$.
\begin{claim} With high probability, we have the following properties.\label{clm:directededges}
\begin{enumerate}[label = {\textbf{E\arabic{enumi}}}]\addtocounter{enumi}{1}
\item For each $v\in V(G)$, if $\psi(v)> 0$ then $v$ has at least $d^{0.8}$ out-neighbours in $Z$ in $\vec{H}_0$.\label{prop:outneighboursplus}
\item For each $v\in V(G)$, if $\psi(v)< 0$ then $v$ has at least $d^{0.8}$ out-neighbours in $Z$ in $\vec{G}\setminus \vec{H}_0$.\label{prop:outneighboursminus}
\end{enumerate}
\end{claim}
\claimproofstart[Proof of Claim~\ref{clm:directededges}] 
Let $v\in V(G)$ with $\psi(v)>0$. By \ref{prop:corr:options1}, we have that $v$ has at least $d/2$ neighbours in $H_0$.
Therefore, by Lemma~\ref{lem:chernoff}, with probability $1-\exp(-\Omega(d))=1-o(n^{-1})$, $v$ has at least $d/5$ 
out-neighbours in $\vec{H}_0$. 
Let $Z_v$ be the set of out-neighbours of $v$ in $\vec{H}_0$ that are contained in $Z$. We will show that $|Z_v|\geq d^{0.8}$ with probability $1-o(n^{-1})$.

Suppose\footnote{This case is contained already within the Fox-Luo-Pham result \cite{fox2024random}, but as it requires only an additional brief, elementary argument we will cover it using our methods as well.} that $d\le n/\log^{3} n$.  The left hand side of \eqref{eq:pickingZ} has mean $O(pt^4n/d)=o(\eps n/(d+1))$, while $pt^4n/d=\eps n/dt\geq \eps \log^2n=\omega(\log n)$. Thus, by Lemma~\ref{lem:chernoff}, with probability $1-o(n^{-1})$ we have  that \eqref{eq:pickingZ} holds for all $i\in [0,d]$ and hence $J=[0,d]$. Therefore,  $Z_v=\{u\in N^+_{\vec{H}_0}(v):y_u\leq p\}$ with probability $1-o(n^{-1})$. 
Hence, as $|\{u\in N^+_{\vec{H}_0}(v):y_u\leq p\}|$ is binomially distributed with mean $p d^+_{\vec{H}_0}(v)\ge pd/5$, where $pd=\eps d/t^5= \omega(d^{0.8})=\omega(\log n)$ (as $d=\omega(\log^{25}n)$), by Lemma~\ref{lem:chernoff} we have that $|Z_v|\geq d^{0.8}$ with probability $1-o(n^{-1})$. 

So, suppose $d> n/\log^{3} n$. For each $u\in V(G)$, changing $y_u$ can change $|Z_v|$ by at most $2(t+2)^4\cdot n/(d+1)$. Indeed, if $i\in [0,d]$ is such that $u\in W_{i}$, then this can only change whether a vertex $w$ is in or out of $Z$ if $w\in \cup_{j\in I+I-I-I}W_{i-j}$, where this set has size at most $(t+2)^4\cdot 2n/(d+1)$. Thus, letting $K=\log^{8}n$, as $d\ge n/\log^{3} n$ implies that $2(t+2)^4\cdot n/(d+1)\leq K$, we have that $|Z_v|$ is $K$-Lipschitz.

Now, consider an out-neighbour $z$ of $v$ in $\vec{H}_0$. We will bound below the probability of ``$z \in Z_v$'', which is equivalent to ``$z\in Z$''.  Let $i= d_{H_0}(z)-\psi(z)$, so that $z\in W_i$. By definition, the event ``$z\in Z$''  is  the  intersection of ``$y_z\le p$'' and ``$|\{u\in \bigcup_{j\in I+I-I-I}W_{i-j}:y_u\leq p\}|\leq \frac{\eps n}{d+1}$''. Since $z\in W_i\subset \bigcup_{j\in I+I-I-I}W_{i-j}$, this is exactly the intersection of the independent events ``$y_z\le p$'' and  ``$|\{u\in \bigcup_{j\in I+I-I-I}(W_{i-j}\setminus \{z\}):y_u\leq p\}|\leq \frac{\eps n}{d+1}-1$''.  Note that the random variable $|\{u\in \bigcup_{j\in I+I-I-I}(W_{i-j}\setminus \{z\}):y_u\leq p\}|$ is binomially distributed with mean $\le p\cdot 2(t+2)^4n/(d+1)=2\eps n(t+2)^4/(t^5(d+1))\le (\frac{\eps n}{d+1}-1 )/2$. By Markov's inequality, we have  $\P(|\{u\in \bigcup_{j\in I+I-I-I}(W_{i-j}\setminus \{z\}):y_u\leq p\}|> \frac{\eps n}{d+1}-1)\le 1/2$. Combining these observations, we get
\begin{align*}
\P(z\in Z_v)&= \P(y_z\leq p)\cdot \P\bigg(\Big|\Big\{u\in \bigcup_{j\in I+I-I-I}(W_{i-j}\setminus \{z\}):y_u\leq p\Big\}\Big|\leq \frac{\eps n}{d+1}-1\bigg)\ge \frac{p}{2}.
\end{align*}
Thus, $\mathbb{E}|Z_v|\ge pd/10$. As $d> n/\log^{3}n$, $K=\log^{8}n$, and $Z_v$ is $K$-Lipschitz, by Lemma~\ref{lem:mcd} with $t'=pd/20$, we have $|Z_v|\geq pd/20\geq d^{0.8}$ with probability $\geq 1-2\exp(-{(pd/20)^2}/{2K^2n}) = 1-o(n^{-1})$. 

Therefore, whether $d\leq n/\log^{3}n$ or $d> n/\log^{3}n$, \ref{prop:outneighboursplus} holds with high probability  by a union bound. That \ref{prop:outneighboursminus} also holds with high probability follows by a symmetric argument.
\claimproofend

Using Claim~\ref{clm:directededges}, take an outcome of $Z$ and $\vec{G}$ for which \ref{prop:outneighboursplus} and \ref{prop:outneighboursminus} hold. Let $G_0=G\setminus H_0$. Let $r=|Z|$ and $Z=\{v_1,\dots,v_r\}$.
We will build graphs $H_0\supset H_1\supset \dots \supset H_r$ and $G_0\supset G_1\supset \dots \supset G_r$, all with vertex set $V(G)$. For each $i\in [0,r]$ and $v\in V(G)$, set $f_i(v)= d_{\vec{H}_0\setminus \vec{H}_{i}}^+(v)-d_{\vec{G}_0\setminus\vec{G}_{i}}^+(v)$, noting that $f_0(v)=0$. 
For each $i\in [r]$ in turn, we build $H_{i}$ and $G_i$ from $H_{i-1}$ and $G_{i-1}$ as follows.
\begin{enumerate}[label = \alph{enumi})]
\item Let $m_i^+$ be the number of vertices $v\in V(G)$ with $\vec{vv_i}\in E(\vec{H}_{i-1})$ and $f_{i-1}(v)<\psi(v)$.\label{step:a} 
\item Let $n_i^+=\max\{s\in I:s\leq m_i^+\}$. Uniformly at random, pick $n_i^+$ vertices $v$ with $\vec{vv_i}\in E(\vec{H}_{i-1})$ and $f_{i-1}(v)<\psi(v)$, remove each edge $\vec{vv_i}$ from $\vec{H}_{i-1}$, and call the resulting oriented graph $\vec{H}_i$.\label{step:b}
\item Let $m_i^-$ be the number of vertices $v\in V(G)$ with $\vec{vv_i}\in E(\vec{G}_{i-1})$ and $f_{i-1}(v)>\psi(v)$.\label{step:c}
\item Let $n_i^-=\max\{s\in I:s\leq m_i^-\}$. Uniformly at random, pick $n_i^-$ vertices $v$ with $\vec{vv_i}\in E(\vec{G}_{i-1})$ and $f_{i-1}(v)>\psi(v)$, remove each edge $\vec{vv_i}$ from $\vec{G}_{i-1}$, and call the resulting oriented graph $\vec{G}_i$.\label{step:d}
\end{enumerate}

The following three properties follow directly from this process.
\begin{enumerate}[label = {\textbf{F\arabic{enumi}}}]
\item For each $v\in V(G)\setminus Z$, we have $d_{\vec{H}_0\setminus \vec{H}_r}^-(v)=d_{\vec{G}_0\setminus \vec{G}_r}^-(v)=0$.\label{prop:final:1}
\item For each $v\in Z$, we have $d_{\vec{H}_0\setminus \vec{H}_r}^-(v)\in I $ and $d_{\vec{G}_0\setminus \vec{G}_r}^-(v)\in I$.\label{prop:final:2}
\item For each $i\in [r]$ and $v\in V(G)$, the following hold. \label{prop:final:2.5}
\begin{itemize}
    \item If $f_{i-1}(v)<\psi(v)$, then $f_{i}(v)\in \{f_{i-1}(v), f_{i-1}(v)+1\}$.
    \item If $f_{i-1}(v)>\psi(v)$, then $f_{i}(v)\in \{f_{i-1}(v), f_{i-1}(v)-1\}$.
    \item If $f_{i-1}(v)=\psi(v)$, then $f_{i}(v)=f_{i-1}(v)$.
\end{itemize}
\end{enumerate}

We now show that the final property we need for $\vec{H}_r$ and $\vec{G}_r$ is likely to hold.

\begin{claim} With high probability, the following holds.  \label{clm:propsofHrGr}
\begin{enumerate}[label = {\textbf{F\arabic{enumi}}}]\addtocounter{enumi}{3}
\item For each $v\in V(G)$, we have $d_{\vec{H}_0\setminus \vec{H}_{r}}^+(v)-d_{\vec{G}_0\setminus\vec{G}_{r}}^+(v)=f_{r}(v)=\psi(v)$.\label{prop:final:3}
\end{enumerate}
\end{claim}
\claimproofstart[Proof of Claim~\ref{clm:propsofHrGr}] Let $v\in V(G)$. If $\psi(v)=0$, then $f_0(v)=\psi(v)$, and \ref{prop:final:2.5} gives $f_i(v)=\psi(v)$ for all $i\in [0,r]$, so that \ref{prop:final:3} holds for $v$. Suppose, then, that $\psi(v)>0$, where the case where $\psi(v)<0$ follows similarly. Since we have $\psi(v)>0=f_0(v)$, it follows from \ref{prop:final:2.5} that  $\psi(v)\ge f_i(v)$ for each $i\in [r]$. Thus, $v$ never has an out-edge deleted at \ref{step:d} for any $i\in [r]$, and so $d^+_{\vec{G}_0\setminus \vec{G}_r}(v)=0$ which gives $f_{r}(v)=d^+_{\vec{H}_0\setminus \vec{H}_r}(v)$.

Let $I_v=\{j:\vec{vv_j}\in E(\vec{H}_0)\}$ so that, by \ref{prop:outneighboursplus}, we have $|I_v|\geq d^{0.8}$.  Let $I_v=\{j_1, \dots, j_{|I_v|}\}$ with $j_1< \dots< j_{|I_v|}$.  For each $j\in I_v$, define an indicator random variable $X_j$ which is $1$ whenever the edge $\vec{vv_{j}}$ is deleted or $f_{j-1}(v)=\psi(v)$, and 0 otherwise. 
We will now show that, for each $i\in [|I_v|]$, $\mathbb E(X_{j_i}|X_{j_1},\dots,X_{j_{i-1}})\ge 1/2$. 

To this end, let $i\in [|I_v|]$, and let $S$ be the set of all $2(j_i-1)$-tuples which can represent in order an outcome for the edge sets chosen at step \ref{step:b} and step \ref{step:d} for each $h\in [j_i-1]$. For each $\mathbf{s}\in S$, let $E_\mathbf{s}$ be the event that these outcomes are chosen, so that the events $E_\mathbf{s}$, $\mathbf{s}\in S$, are disjoint and partition the whole space of possibilities. 
Then, letting $j=j_i$, $\P(X_{j}=1| E_\mathbf{s})$ can be bounded below as follows. Notice that the edge $\vec{vv_{j}}$ can only be deleted when $\vec{H}_{j}$ is created from $\vec{H}_{{j}-1}$. Therefore, if $\mathbf{s}$ is such that $f_{j-1}(v)<\psi(v)$ when $E_\mathbf{s}$ holds, then the probability that the edge $\vec{vv_{j}}$ is deleted is ${n^+_{j}}/{m^+_{j}}\geq 1/2$, i.e., $\P(X_{j}=1|E_\mathbf{s})\ge 1/2$. Also, if $\mathbf{s}$ is such that $f_{j-1}(v)=\psi(v)$, then $\P(X_{j}=1|E_\mathbf{s})=1\ge 1/2$.
Let $\mathbf{X}=(X_{j_1},\dots,X_{j_{i-1}})$. For each $\mathbf{x}\in \{0,1\}^{i-1}$ with $\mathbb{P}(\mathbf{X}=\mathbf{x})>0$, let $S_\mathbf{x}\subset S$ be the set of outcomes of choices before step $j_i$ for which ``$\mathbf{X}=\mathbf{x}$'' occurs. By the law of total probability, then,  
\[
\mathbb E(X_{j_i}|\mathbf{X}=\mathbf{x})=\P(X_{j_i}=1|\mathbf{X}=\mathbf{x})
=\sum_{s\in S_\mathbf{x}}\P(X_{j_i}=1|E_s)\cdot \P(E_s|\mathbf{X}=\mathbf{x})
\ge \sum_{s\in S_\mathbf{x}}\frac{1}{2}\cdot \P(E_s|\mathbf{X}=\mathbf{x})=\frac{1}{2}.
\]
Thus, we have $\mathbb E(X_{j_i}|X_{j_1},\dots,X_{j_{i-1}})\ge 1/2$ for each $i\in [|I_v|]$.

From \ref{prop:outneighboursplus} we had $|I_v|\geq d^{0.8}$, so, by Lemma~\ref{lem:azuma}, we thus have
$$\P\bigg(\sum_{j\in I_v}X_j\le \frac{|I_v|}{4}\bigg) \le \P\bigg(\sum_{j\in I_v}X_j\le \frac{|I_v|}{2}-\frac{|I_v|}{4}\bigg)\le \exp\left(\frac{-(|I_v|/4)^2}{2|I_v|/2+2|I_v|/4}\right)\le \exp\left(-\frac{d^{0.8}}{32}\right).$$
As $d=\omega(\log^{25}n)$, and $\psi(v)\leq d^{3/4}$ from \ref{prop:corr:smallcorrs}, with probability $1-o(n^{-1})$ we have $\sum_{j\in I_v}X_j> |I_v|/{4}\geq  d^{0.8}/4> \psi(v)\ge f_{r}(v)=d^+_{\vec{H}_0\setminus \vec{H}_r}(v)$. From the definition of  $X_{j_1}, \dots, X_{j_{|I_v|}}$, either $f_{i}(v)=\psi(v)$ for some $i\in [r]$  (which implies that $f_{r}(v)=\psi(v)$ via \ref{prop:final:2.5}), or else $\sum_{j\in I_v}X_j$ equals the number of out-edges deleted at $v$, i.e., $d^+_{\vec{H}_0\setminus \vec{H}_r}(v)=\sum_{j\in I_v}X_j$. The latter case contradicts  $\sum_{j\in I_v}X_j>d^+_{\vec{H}_0\setminus \vec{H}_r}(v)$, and so we always have $f_{r}(v)=\psi(v)$.
Thus, by a union bound over $v\in V(G)$, we have that \ref{prop:final:3} holds with high probability.
\claimproofend

Therefore, we can assume that \ref{prop:final:1}--\ref{prop:final:3} hold. Let $H=H_r+(G_0\setminus G_r)$ and let $V_i=\{v\in V(G):d_H(v)=i\}$ for each $i\in [0,d]$. We will show that $H\subset G$ has the property we want. Indeed, as $H_r\subset H_0$ is disjoint from $G_r\subset G_0=G\setminus H_0$, we have, for each $i\in [0,d]$ and $v\in W_i\setminus Z$, that
\[
d_H(v)=d_{H_0}(v)-d^+_{\vec{H}_0\setminus \vec{H}_r}(v)-d^-_{\vec{H}_0\setminus \vec{H}_r}(v)+d^+_{\vec{G}_0\setminus \vec{G}_r}(v)+d^-_{\vec{G}_0\setminus \vec{G}_r}(v)\overset{\ref{prop:final:1},\ref{prop:final:3}}{=}d_{H_0}(v)-\psi(v)=i.
\]
Thus, $W_i\setminus Z\subset V_i$. From \ref{prop:WZsmall} (using that $0\in I-I$), for each $i\in [0,d]$, $|W_i\cap Z|\leq \eps n/(d+1)$ and, hence, by \ref{prop:corr:goodpots}, $|V_i|\geq |W_i|-|Z\cap W_i|\geq (1-2\eps)n/(d+1)$.

Now, observe that, for each $i\in [0,d]$ and $v\in W_i\cap Z$,
\begin{align*}
d_H(v)&=d_{H_0}(v)-d^+_{\vec{H}_0\setminus \vec{H}_r}(v)-d^-_{\vec{H}_0\setminus \vec{H}_r}(v)+d^+_{\vec{G}_0\setminus \vec{G}_r}(v)+d^-_{\vec{G}_0\setminus \vec{G}_r}(v)\\
&\overset{\ref{prop:final:3}}{=}d_{H_0}(v)-\psi(v)-d^-_{\vec{H}_0\setminus \vec{H}_r}(v)+d^-_{\vec{G}_0\setminus \vec{G}_r}(v)=i-d^-_{\vec{H}_0\setminus \vec{H}_r}(v)+d^-_{\vec{G}_0\setminus \vec{G}_r}(v)\overset{\ref{prop:final:2}}{\in} \{i\}-I+I.
\end{align*}
Thus, for each $i\in [0,d]$, using \ref{prop:corr:goodpots} and \ref{prop:WZsmall} we get
\[
|V_i|\leq |W_i|+\sum_{j\in I-I}|Z\cap W_{i-j}|\leq \frac{(1+2\eps)n}{d+1}.
\]
Therefore, for each $i\in [0,d]$, there are $(1\pm 2\eps)n/(d+1)$ vertices in $H$ with degree $i$, as required.
\end{proof}


\end{document}